\documentclass{amsart}
\usepackage{enumerate,amssymb}
\usepackage{color,comment}

\theoremstyle{plain}
\newtheorem{thm}{Theorem}
\newtheorem{lem}[thm]{Lemma}
\newtheorem{cor}[thm]{Corollary}
\newtheorem{prop}[thm]{Proposition}

\theoremstyle{remark}
\newtheorem{rem}[thm]{Remark}
\newtheorem{construct}[thm]{Construction}

\begin{document}

\title{A decomposition theorem in II$_1$--factors}

\author[Dykema]{K. Dykema$^*$}
\address{Department of Mathematics, Texas A\&M University, College Station, TX, USA.}
\email{ken.dykema@math.tamu.edu}
\thanks{\footnotesize ${}^{*}$ Research supported in part by NSF grant DMS--1202660.}
\author[Sukochev]{F. Sukochev$^{\S}$}
\address{School of Mathematics and Statistics, University of new South Wales, Kensington, NSW, Australia.}
\email{f.sukochev@math.unsw.edu.au}
\thanks{\footnotesize ${}^{\S}$ Research supported by ARC}
\author[Zanin]{D. Zanin$^{\S}$}
\address{School of Mathematics and Statistics, University of new South Wales, Kensington, NSW, Australia.}
\email{d.zanin@math.unsw.edu.au}



\begin{abstract} 
Building on results of Haagerup and Schultz, we decompose an arbitrary operator in a diffuse, finite von Neumann algebra
into the sum of a normal operator and an s.o.t.-quasinilpotent operator. We also prove an analogue of Weyl's inequality relating eigenvalues and singular values for operators in a diffuse, finite von Neumann algebra.
\end{abstract}
\maketitle

\section{Introduction}

The following result is due to Schur (see e.g. \cite{Zhang}) and is one of the cornerstones of linear algebra.
\begin{thm} For every matrix $T\in M_n(\mathbb{C}),$ there is a unitary matrix $U\in M_n(\mathbb{C})$ such that $U^{-1}TU$ is an upper-triangular matrix.
\end{thm}

Alternatively, there exists a basis in $\mathbb{C}^n$ such that the matrix of the operator $T$ with respect to this basis is upper-triangular. Taking the diagonal part $N$ of the operator $T$ in this basis, we obtain a normal operator. The difference $T-N$ is, obviously, a strictly upper-triangular matrix. Every strictly upper-triangular matrix is, clearly, nilpotent. Thus, an arbitrary matrix is a sum of a normal matrix and nilpotent matrix.

The following theorem due to Ringrose \cite{Ringrose} extends Schur's result to the realm of compact operators in Hilbert space. 
Recall that an operator $Q$ is quasinilpotent if its spectrum is $\{0\}$
and that, by the C$^*$--property, the functional calculus for self-adjoint operators and the spectral radius formula
 (see~\cite{KR1}),
for an operator $Q$ on Hilbert space we have 
\[
\|\big((Q^*)^nQ^n\big)^{1/2n}\|_\infty=\|(Q^*)^nQ^n\|_\infty^{1/2n}=\|Q^n\|_\infty^{1/n}
\]
and quasinilpotency of $Q$
is equivalent to
\begin{equation}\label{eq:Q*nQn}
\lim_{n\to\infty}\|\big((Q^*)^nQ^n)^{1/2n}\|_{\infty}=0.
\end{equation}
(Here and throughout this paper, we use the notation $\|X\|_\infty$ for the operator norm of a bounded operator $X$ on Hilbert space.)

\begin{thm}\label{ringrose} For every compact operator $T\in B(H),$ there exists an increasing net of projections $p_{\lambda}$,
$\lambda\in[0,1]$, with $p_0=0$ and $p_1=1$, such that, letting $p_{\lambda-0}=\vee_{\mu<\lambda}\,p_\mu,$
\begin{enumerate}[{\rm (a)}]
\item\label{rina} $Tp_{\lambda}=p_{\lambda}Tp_{\lambda}$ for all $\lambda\in[0,1].$
\item\label{rinb} for every $\lambda\in(0,1]$ either $p_{\lambda}=p_{\lambda-0}$ or $p_{\lambda}-p_{\lambda-0}$ is a one-dimensional projection.
\end{enumerate}
Furthermore, if for such a family we have $Tp_{\lambda}=p_{\lambda-0}Tp_{\lambda}$ for all $\lambda\in(0,1],$ then $T$ is quasinilpotent.
\end{thm}

This yields as an immediate corollary the following decomposition result:

\begin{cor}\label{cor:ringrose} For every compact operator $T\in B(H),$ there exist a normal operator $N$ and a quasinilpotent operator $Q$ such that $T=N+Q.$
\end{cor}
\begin{proof} Indeed, set
$$N=\sum_{p_{\lambda}\neq p_{\lambda-0}}(p_{\lambda}-p_{\lambda-0})T(p_{\lambda}-p_{\lambda-0}),\quad Q=T-N,$$
where the sum (of pairwise orthogonal $1-$dimensional operators)  converges in strong operator topology.
It is clear that $N$ is normal and that $Qp_{\lambda}=p_{\lambda-0}Qp_{\lambda}.$ In particular, it follows from Theorem \ref{ringrose} that $Q$ is quasinilpotent.
\end{proof}

It is quite natural to ask whether the latter decomposition remains true in other settings.
In this paper, we concentrate on II$_1$--factors and, more generally, diffuse, finite von Neumann algebras
(see \S\ref{subsec:II1}).

The following result due to Haagerup and Schultz
\cite{HS2} is of utmost importance for our investigation.
(See \S\ref{subsec:FK} below for description of the Brown measure.)

\begin{thm}\label{thm:HS} Let $\mathcal{M}\subseteq B(H)$ be a II$_1$--factor with tracial state $\tau$
and let $T\in\mathcal{M}.$ For every Borel set $\mathcal{B}\subset\mathbb{C},$ there exists a unique projection $p_{\mathcal{B}}\in\mathcal{M}$ such that
\begin{enumerate}[{\rm (a)}]
\item\label{hsa} $\tau(p_{\mathcal{B}})=\nu_T(\mathcal{B})$, where $\nu_T$ is the Brown measure of $T$
\item\label{hsb} $Tp_{\mathcal{B}}=p_{\mathcal{B}}Tp_{\mathcal{B}}$
\item\label{hsc} if $p_{\mathcal{B}}\ne0$, then the Brown measure of $Tp_{\mathcal{B}}$ considered as an element of $p_{\mathcal{B}}\mathcal{M}p_{\mathcal{B}}$ is supported in $\mathcal{B}$
\item\label{hsd} if $p_{\mathcal{B}}\ne1,$ then the Brown measure of $(1-p_{\mathcal{B}})T$ considered as an element of $(1-p_{\mathcal{B}})\mathcal{M}(1-p_{\mathcal{B}})$ is supported in $\mathbb{C}\backslash\mathcal{B}.$
\end{enumerate}
Moreover, $p_{\mathcal{B}}$ is $T$-hyperinvariant, meaning that it is invariant under every $S\in B(H)$ that commutes with $T$.
If Borel sets $\mathcal{B}_1,\mathcal{B}_2\subset\mathbb{C}$ are such that $\mathcal{B}_1\subset\mathcal{B}_2,$ then $p_{\mathcal{B}_1}\leq p_{\mathcal{B}_2}.$
\end{thm}

This projection $p_{\mathcal{B}}$ is called the Haagerup--Schultz projection for the operator $T$ associated to the set $\mathcal{B}$.

\begin{rem}
In the above theorem, the hypothesis that $\mathcal{M}$ be a II$_1$--factor and $\tau$ its tracial state
can clearly be loosened to require only that
$\mathcal{M}$ be a diffuse, finite von Neumann algebra and $\tau$ be any normal, faithful, tracial state on it.
This is because (a) any such $\mathcal{M}$ can be embedded, via a normal, trace-preserving $*$-homomorphism, into a II$_1$--factor
(to see this, use for example Lemma~2.3 and Theorem~2.5 from~\cite{D94} to see that the free product
with respect to traces of $\mathcal{M}$ with $L^\infty[0,1]$ is a II$_1$-factor)
and (b) the projection onto a hyperinvariant subspace of any operator belongs to the von Neumann algebra generated by the operator
(this is easy to show, but for a proof see, for example, Proposition~ 2.1 of~\cite{D05}).
\end{rem}

The purpose of this paper is to use the Haagerup--Schultz theorem above to obtain the following
finite von Neumann algebra version of the Ringrose result.
We note that the quasinilpotent operator in Ringrose's theorem is here replaced by an s.o.t.-quasinilpotent operator, which is
an operator $Q$ such that $((Q^*)^nQ^n)^{1/2n}$ converges in strong operator topology to $0$;
in this notation (which was introduced in~\cite{DFS}) s.o.t.\ is an abbreviation for ``strong operator topology;''
that quasinilpotent operators must be s.o.t.--quasinilpotent follows from the characterization~\eqref{eq:Q*nQn}
of quasinilpotency.
By Theorem~8.1 of~\cite{HS2}, in a finite von Neumann algebra the s.o.t.-quasinilpotent operators are precisely those whose
Brown measures are concentrated at $\{0\}$.
\begin{thm}\label{decomposition} Let $\mathcal{M}$ be a diffuse, finite von Neumann algebra with
normal, faithful, tracial state $\tau$ and let $T\in\mathcal{M}$.
Then there exist $N,Q\in\mathcal{M}$ such that
\begin{enumerate}[{\rm (a)}]
\item\label{deca} $T=N+Q$,
\item\label{decb} the operator $N$ is normal and the Brown measure of $N$ equals that of $T$,
\item\label{decc} the operator $Q$ is s.o.t-quasinilpotent.
\end{enumerate}
\end{thm}

As in Ringrose's theorem, a family $(q_t)_{0\le t\le 1}$ of $T$--invariant (in fact, $T$--hyper\-in\-vari\-ant) projections
is involved in Theorem~\ref{decomposition}, and $N$ and $Q$ can be regarded as diagonal and upper triangular, respectively,
with respect to this family of projections.
Indeed, $N$ is obtained as the conditional expectation of $T$ onto the von Neumann algebra generated by $\{q_t\mid 0\le t\le 1\}$.
However, unlike in Ringrose's theorem, the dimensions (i.e., in our finite von Neumann algebra setting, the traces)
of the projections $q_t$ can take large jumps as $t$ varies.
Indeed, if $T$ itself is s.o.t.-quasinilpotent, then our construction yields $N=0$ and $\{q_t\mid 0\le t\le 1\}\subseteq\{0,1\}$.

\medskip

For compact operators in $B(H),$ it was shown in \cite{Ringrose} that in Corollary~\ref{cor:ringrose},
the eigenvalues of the operator $N$ coincide with those of the operator $T$ (counting the algebraic multiplicities). The analogous result for the operators in diffuse, finite von Neumann algebras is given by Theorem \ref{decomposition} \eqref{decb}. 

The eigenvalues $\lambda(k,T),$ $k\geq0,$ and singular values $\mu(k,T),$ $k\geq0,$ of a compact operator $T\in B(H)$ are related by means of the Weyl theorem \cite{Weyl}. See \cite{GK1} for detailed proof.

\begin{thm}\label{weyl orig} Let $T\in B(H)$ be a compact operator and let $\Phi$ be a real--valued increasing function on $[0,\infty)$
such that $\Phi(0)=0$ and $\Phi\circ\exp$ is convex. Then
$$\sum_{k=0}^{\infty}\Phi(|\lambda(k,T)|)\leq\sum_{k=0}^{\infty}\Phi(\mu(k,T)).$$
\end{thm}

We will also prove the following theorem, which extends Weyl's result to the II$_1$--setting:
\begin{thm}\label{thm:WeylII1} Let $\mathcal{M}$ be 
a diffuse finite von Neumann algebra with
normal, faithful, tracial state $\tau$
and let $T\in\mathcal{M}.$ For the normal operator $N$ from Theorem \ref{decomposition} and for every 
real-valued increasing function $\Phi$ on $[0,\infty)$ such that $\Phi\circ\exp$ is convex, we have
$$ \int\Phi(|z|)\,d\mu_T(z)=\tau(\Phi(|N|))\leq\tau(\Phi(|T|)).$$ 
\end{thm}

\section{Preliminaries}

\subsection{Finite von Neumann algebras and II$_1$--factors}
\label{subsec:II1}

A von Neumann algebra $\mathcal{M}$ is called
{\em finite} if it has a normal, faithful, tracial state $\tau$.
It is called {\em diffuse} if it has no minimal (nonzero) projections and it is called 
a {\em factor} if its center is trivial.
The infinite dimensional finite von Neumann algebra factors are diffuse and are called II$_1$--factors, and each of these
has a unique tracial state $\tau$, which is normal and faithful.
See, e.g., \cite{KR1} for details.

Throughout this paper, $\mathcal{M}$ will denote a diffuse, finite von Neumann algebra and $\tau$ will be a normal, faithful, tracial state
on $\mathcal{M}$.

\subsection{Singular value function}

For every $T\in\mathcal{M},$ the generalised singular value function $\mu(T)$, denoted $t\to\mu(t,T)$ for $t\in(0,1)$,
is defined by the formula (see, e.g., \cite{FackKosaki})
$$\mu(t,T)=\inf\{\|Tp\|_{\infty}:\ p\in\operatorname{Proj}(\mathcal{M}),\;\tau(1-p)\leq t\}.$$
It is continuous from the right in $t.$ Equivalently, $\mu(T)$ can be defined in terms of the distribution function $d_{|T|}$ of the operator $|T|.$ That is, setting
$$d_{|T|}(s)=\tau(E^{|T|}(s,\infty)),\quad s\geq0,$$
we obtain
$$\mu(t,T)=\inf\{s\geq0:\ d_{|T|}(s)\leq t\},\quad t>0.$$
Here, $E^{|T|}$ denotes the projection valued spectral measure of the operator $|T|.$

The following result is a widely known consequence of the spectral theorem.

\begin{lem}\label{boolean lemma} Let $\mathcal{M}$ be a diffuse, finite von Neumann algebra equipped with a normal, faithful, tracial state $\tau$ and let $0\leq A\in\mathcal{M}$.
Then there is an increasing net $(p_s)_{0\le s\le1}$ of projections in $\mathcal{M}$ with $\tau(p_s)=s$
with
\[
A=\int_0^1 \mu(s,A)\,dp_s.
\]
\end{lem}
\begin{proof}[Sketch of proof]
If $s$ is such that $\mu(s,A)$ is a point of continuity of the distribution function $d_A$,
then let $p_s$ be the spectral projection $E^A((\mu(s,A),\infty))$.
At any remaining points $s$, the projection $E^A(\{\mu(s,A)\})$ is nonzero;
there are at most countably many values $r$ where $E^A(\{r\})$ is nonzero;
for each of them, let $a(r)=\tau(E^A(\{r\})$ and
choose an increasing family $(q^{(r)}_t)_{0\le t\le a(r)}$ of projections with $\tau(q^{(r)}_t)=t$ and $q^{(r)}_{a(r)}=E^A(\{r\})$;
when $\mu(s,A)=r$ is one of these points,
let $p_s=E^A((\mu(s,A),\infty))+q^{(r)}_t$ with $t$ chosen so that $\tau(p_s)=s$.
\end{proof}

\subsection{Fuglede-Kadison determinant and Brown measure}
\label{subsec:FK}

Fuglede and Kadison~\cite{FK} constructed a mapping $\Delta:\mathcal{M}\to\mathbb{R}_+$ which is a homomorphism with respect to the multiplication. This mapping is defined by
\begin{equation}\label{det}
\Delta(T)=\exp(\tau(\log(|T|))),\quad T\in\mathcal{M}.
\end{equation}
For every operator $T,$ the function
\begin{equation}\label{subhar}
\lambda\to\tau(\log(|T-\lambda|)),\quad\lambda\in\mathbb{C},
\end{equation}
is shown to be subharmonic by Brown \cite{Brown}. Using this fact, Brown constructed a probability measure $\nu_T$ such that
\begin{equation}\label{br def}
\tau(\log(|T-\lambda|))=\int_{\mathbb{C}}\log(|z-\lambda|)d\nu_T(z)\quad \lambda\in\mathbb{C}.
\end{equation}
This $\nu_T$, (called the Brown measure of $T$) can be viewed as the II$_1$--analogue of the spectral counting measure (according
to algebraic multiplicity) on matrices. It can be recovered by taking the Laplacian of the mapping in \eqref{subhar}.

The following is Proposition 2.24 of \cite{HS1} and a consequence of it.
\begin{thm}\label{thm:det nu matrix}
If $T\in\mathcal{M}$ and if $p\in\mathcal{M}$ is a projection such that $Tp=pTp,$ so that we may write
$T=\left(\begin{smallmatrix}
A&B\\
0&C
\end{smallmatrix}\right),$
where $A=Tp$ and $C=(1-p)T,$ then
\begin{equation}\label{eq:DeltaT}
\Delta_{\mathcal{M}}(T)=\Delta_{p\mathcal{M}p}(A)^{\tau(p)}\Delta_{(1-p)\mathcal{M}(1-p)}(C)^{\tau(1-p)}
\end{equation}
and
$$\nu_T=\tau(p)\nu_A+\tau(1-p)\nu_C.$$
\end{thm}
We will use the equation~\eqref{eq:DeltaT} also in the case of $p=0$ or $p=1,$ by making the convention $\Delta_{\{0\}}(0)^0=1.$

\subsection{Haagerup--Schultz projections and s.o.t.--quasinilpotent operators}
It is proved in \cite{HS2} that, for every $T\in\mathcal{M},$ $((T^*)^nT^n)^{1/2n}$ converges as $n\to\infty$ in strong operator topology. The spectral projection of the limiting operator on the interval $[0,r]$ is exactly the Haagerup--Schultz projection $p_{\mathcal{B}_r}$ from Theorem~\ref{thm:HS} corresponding to the ball $\mathcal{B}_r=\{|z|\leq r\}.$ Thus, an operator $T$ has Brown measure supported on $\{0\}$ if and only if $((T^*)^nT^n)^{1/2n}$ converges in strong operator topology to $0.$ We call such operators s.o.t.-quasinilpotent (this notation was introduced in \cite{DFS}).

Though aesthetically attractive, the above definition of the projection $p_{\mathcal{B}_r}$ is not suitable for our purposes. We employ a different characterization also taken from \cite{HS2}. Define the subspace $H_r$ of the Hilbert space $H$ by setting
\begin{equation}\label{eq:H-S H_r}
H_r=\{\xi\in H:\ \exists\xi_n\to\xi,\text{ with }\ \limsup_{n\to\infty}\|T^n\xi_n\|^{1/n}\leq r\}.
\end{equation}
The projection onto the subspace $H_r$ is shown in \cite{HS2} to be the Haagerup--Schultz projection $p_{\mathcal{B}_r}$ corresponding to the ball $\mathcal{B}_r=\{|z|\leq r\}.$

We will not need the construction of Haagerup--Schultz projections for sets other than balls; see \cite{HS2} for the construction in the general case.

\subsection{Submajorization and logarithmic submajorization}

The operator $B\in\mathcal{M}$ is said to be submajorized by the operator $A\in\mathcal{M}$ (written $B\prec\prec A$) if
$$\int_0^t\mu(s,B)ds\leq\int_0^t\mu(s,A)ds,\qquad 0<t<1.$$

The importance of submajorization can be observed from the following theorem, which is really a result about functions rather than operators and is
essentially an inequality of Hardy, Littlewood and Polya (see e.g Lemma II.3.4 of \cite{GK1} for the sequence version, or Proposition 14.H.1.a of \cite{MOA} for a result that implies the following).

\begin{thm}\label{hlp theorem} If $A,B\in\mathcal{M}$ and if $B\prec\prec A,$ then for every increasing convex function $\Phi$ on $[0,\infty),$ we have
$$\tau(\Phi(|B|))\leq\tau(\Phi(|A|)).$$
\end{thm}

We also need the notion of logarithmic submajorization. The operator $B\in\mathcal{M}$ is said to be logarithmically submajorized by the operator $A\in\mathcal{M}$ (written $B\prec\prec_{\log}A$) if
$$\int_0^t\log(\mu(s,B))ds\leq\int_0^t\log(\mu(s,A))ds,\qquad 0<t<1.$$

We collect some easy observations into a lemma, for future use.
\begin{lem}\label{lem:log invertibles}
If $A,B\in\mathcal{M}$ and if $c>0$, then 
\begin{align*}
B\prec\prec A \qquad&\Leftrightarrow\qquad cB\prec\prec cA, \\
B\prec\prec_{\log}A \qquad&\Leftrightarrow\qquad cB\prec\prec_{\log}cA. 
\end{align*}
Furthermore,
if $1\le A,B\in\mathcal{M}$, where $1$ represents the identity operator,
then
\[
B\prec\prec_{\log} A\qquad\Leftrightarrow\qquad \log(B)\prec\prec\log(A).
\]
\end{lem}
\begin{proof}
The first assertion follows from the fact that $\mu(s,cA)=c\mu(s,A)$ and the second from the fact that,
for $A\ge1$, we have $\mu(s,\log(A))=\log(\mu(s,A))$.
\end{proof}

\subsection{Conditional expectation}

If $\mathcal{D}$ is a von Neumann subalgebra of the finite von Neumann algebra $\mathcal{M}$
with normal, faithful, tracial state $\tau$,
then there exists a unique linear operator ${\rm Exp}_{\mathcal{D}}:\mathcal{M}\to\mathcal{D}$ such that,  for all $A\in\mathcal{M}$ and $B\in\mathcal{D}$,
\begin{enumerate}[{\rm (a)}]
\item ${\rm Exp}_{\mathcal{D}}(AB)={\rm Exp}_{\mathcal{D}}(A)B$
\item ${\rm Exp}_{\mathcal{D}}(BA)=B{\rm Exp}_{\mathcal{D}}(A)$
\item $\tau({\rm Exp}_{\mathcal{D}}(A))=\tau(A)$.
\end{enumerate}
Furthermore, ${\rm Exp}_{\mathcal{D}}$ is positive (in fact, completely positive), of norm $1$ and can be realised as the orthogonal projection from $\mathcal{L}_2(\mathcal{M})$ onto $\mathcal{L}_2(\mathcal{D})$ restricted to $\mathcal{M}$.
See, for example, \cite{SS} for these and other facts.

It is well known and not difficult to verify that for the action of $\mathcal{M}$ on $\mathcal{L}_2(\mathcal{M})$,
the strong operator topology on bounded sets
of $\mathcal{M}$ coincides with the topology provided by the norm $\|x\|_2=\tau(x^*x)^{1/2}$.
From this, it is easy to prove the following well known lemma:
\begin{lem}\label{exp limit lemma} Assume $\mathcal{A}_n,$ $n\geq1,$ is a family of von Neumann subalgebras in $\mathcal{M},$
that is either increasing or decreasing in $n.$ Let $\mathcal{A}$ be the strong operator closure of $\bigcup_{n=1}^\infty\mathcal{A}_n$ in the first case
and $\mathcal{A}=\bigcap_{n=1}^\infty\mathcal{A}_n$ in the second case. Then
$${\rm Exp}_{\mathcal{A}}(T)=\lim_{n\to\infty}{\rm Exp}_{\mathcal{A}_n}(T),\quad(T\in\mathcal{M}),$$
where the limit is taken in the strong operator topology.
\end{lem}

\section{Construction of the normal part}

Throughout this section, $\mathcal{M}$ will be a diffuse, finite von Neumann algebra and $\tau$ a
normal, faithul, tracial state on $\mathcal{M}$.
Our plan for proving Theorem \ref{decomposition} is to take as normal operator $N={\rm Exp}_{\mathcal{D}}(T)$ for a suitable commutative von Neumann subalgebra $\mathcal{D}$, namely, the one given below.

\begin{construct}\label{peano} Let $T\in\mathcal{M}$ and let $\rho:[0,1]\to\{|z|\leq \|T\|_{\infty}\}$ be a Peano curve.\footnote{A continuous surjective mapping $[0,1]\to[0,1]^2$ was first constructed by Peano \cite{Peano}. See, for example, \cite{Man} for details. We may take the ball rather than the square since they are homeomorphic.}
\begin{enumerate}[(a)]
\item Set $q_t$ to be the Haagerup--Schultz projection for $T$ associated to the Borel set $\rho([0,t]).$
Then $q_t$ is increasing in $t$.
Since $\tau(q_t)=\nu_T(\rho([0,t])$, we have $q_t=\wedge_{t'>t}q_{t'}$;
i.e., $q_t$ is strong-operator-topology continuous from the right in $t$.
\item Set $\mathcal{D}$ to be the von Neumann algebra generated by $\{q_t\mid t\in[0,1]\}.$
\item For every $n\geq0,$ set $\mathcal{D}_n$ to be the algebra generated by $q_{k/2^n},$ $0\leq k\leq 2^n.$
\end{enumerate}
For technical convenience, we will assume that the Brown measure of $T$ has no atom at $\rho(0)$ (i.e. $\nu_T(\rho(0))=0$). This ensures $q_0=0$ and it can always be arranged by modification of $\rho,$ if necessary. 
\end{construct}

Since the function $\rho$ is uniformly continuous, it follows that there exists a monotone function $\omega:[0,1]\to\mathbb{R}_+$ (called the modulus of continuity of $\rho$) such that $\omega(+0)=0$ and such that
$$|\rho(t_1)-\rho(t_2)|\leq\omega(|t_1-t_2|),\quad t_1,t_2\in[0,1].$$

\begin{lem}\label{lem:Enormconv}
Let $T\in\mathcal{M}$ and let $q_t$, $\mathcal{D}$ and $\mathcal{D}_n$ be as in Construction \ref{peano}.
Then ${\rm Exp}_{\mathcal{D}_n}(T)$ converges in norm to ${\rm Exp}_{\mathcal{D}}(T)$, and, in fact, we have
$$\|{\rm Exp}_{\mathcal{D}_n}(T)-{\rm Exp}_{\mathcal{D}}(T)\|_{\infty}\le\omega(2^{-n}),$$
where $\omega$ is the modulus of continuity of $\rho.$
\end{lem}
\begin{proof} By Theorem \ref{thm:HS}, the projections $q_{k/2^n},$ $0\leq k\leq 2^n,$ are increasing in $k.$ Letting $f^n_k=q_{(k+1)/2^n}-q_{k/2^n}$ for $0\le k<2^n,$ we have
\begin{equation}\label{eq:EDn}
{\rm Exp}_{\mathcal{D}_n}(T)=\sum_{\substack{0\le k<2^n \\ f^n_k\ne0}}\frac{\tau(f^n_kTf^n_k)}{\tau(f^n_k)}f^n_k.
\end{equation}
By Construction \ref{peano} and Theorem \ref{thm:HS}, when $f^n_k\ne0,$ the Brown measure of $f^n_kTf^n_k$ in $f^n_k\mathcal{M}f^n_k$ is supported in $\rho([0,\frac{k+1}{2^n}])\backslash \rho([0,\frac{k}{2^n}])$ (which is a subset of $\rho((\frac{k}{2^n},\frac{k+1}{2^n}])$). It follows from Theorem \ref{thm:HS} \eqref{hsc},\eqref{hsd} combined with Brown's analogue of Lidskii's theorem \cite{Brown} that
\begin{equation}\label{trace partitioning}
\frac{\tau(f^n_kTf^n_k)}{\tau(f^n_k)}=\frac{\int_{\rho([0,\frac{k+1}{2^n}])\backslash \rho([0,\frac{k}{2^n}])}zd\nu_T(z)}{\nu_T(\rho([0,\frac{k+1}{2^n}])\backslash \rho([0,\frac{k}{2^n}]))}\in\overline{\rm conv}(\rho((\frac{k}{2^n},\frac{k+1}{2^n}])).
\end{equation}
Now take $m>n$ and note that
$$f^n_k=\sum_{j=2^{m-n}k}^{2^{m-n}(k+1)-1}f_j^m.$$
For $2^{m-n}k\leq j< 2^{m-n}(k+1)$ such that $f^m_j\ne0$, we have
$$\frac{\tau(f^m_jTf^m_j)}{\tau(f^m_j)}\in\overline{\rm conv}(\rho((\frac{j}{2^m},\frac{j+1}{2^m}]))\subset\overline{\rm conv}(\rho((\frac{k}{2^n},\frac{k+1}{2^n}])).$$
It follows that
\begin{equation}\label{eq:EDnDm}
\|{\rm Exp}_{\mathcal{D}_n}(T)-{\rm Exp}_{\mathcal{D}_m}(T)\|_{\infty}\le\max_{0\le k<2^n}\operatorname{diam}\big(\rho((\frac{k}{2^n},\frac {k+1}{2^n}])\big)\le\omega(2^{-n})
\end{equation}
and this upper bound tends to $0$ as $n\to\infty.$
By Lemma \ref{exp limit lemma}, ${\rm Exp}_{\mathcal{D}_n}(T)$ converges in strong operator topology to ${\rm Exp}_{\mathcal{D}}(T).$ By the above estimate, it is Cauchy in the uniform norm, and, therefore, converges in that norm to ${\rm Exp}_{\mathcal{D}}(T).$ Now letting $m\to\infty$ in \eqref{eq:EDnDm} completes the proof of the lemma.
\end{proof}

\begin{lem}\label{det convergence} Let $T\in\mathcal{M}$ and let $q_t$,  $\mathcal{D}$ and $\mathcal{D}_n$ be as in Construction \ref{peano}. For every $\lambda\in\mathbb{C},$ and $\varepsilon>0$, we have
\begin{equation}\label{eq:det convergence}
\lim_{n\to\infty}\log\Delta(|{\rm Exp}_{\mathcal{D}_n}(T)-\lambda|^2+\varepsilon)=\int_{\mathbb{C}}\log(|z-\lambda|^2+\varepsilon)d\nu_T(z).
\end{equation}
\end{lem}
\begin{proof} Letting $f^n_k$ be as in the proof of Lemma~\ref{lem:Enormconv} and using~\eqref{eq:EDn}, we have
\begin{equation}\label{eq:logEDn}
\log\Delta(|{\rm Exp}_{\mathcal{D}_n}(T)-\lambda|^2+\varepsilon)=\sum_{\substack{0\le k<2^n \\ f^n_k\ne0}}\tau(f^n_k)\log(|\frac{\tau(f^n_kTf^n_k)}{\tau(f^n_k)}-\lambda|^2+\varepsilon).
\end{equation}
We now define for each $n$ a function $x_n$ on the disk $\{|z|\le\|T\|_\infty\}$.
Letting $z$ be a complex number with 
$|z|\le\|T\|_\infty$ and $z\ne\rho(0)$, we have
$z\in\rho([0,\frac{k+1}{2^n}])\backslash \rho([0,\frac{k}{2^n}])$ for some unique $k=k(z)\in\{0,\ldots,2^n-1\}$.
Indeed, selecting the minimal $t$ such that $z=\rho(t)$, we take $k$ such that $k/2^n<t\leq(k+1)/2^n$.
In the case $f^n_k\ne0$, we let
\begin{equation}\label{xn def}
x_n(z)=\log(|\frac{\tau(f^n_kTf^n_k)}{\tau(f^n_k)}-\lambda|^2+\varepsilon), 
\end{equation}
while if $z=\rho(0)$ or $f^n_{k(z)}=0$, then for specificity we set $x_n(z)=\log\varepsilon$.
(Note, however, that the set $\{\rho(0)\}\cup\{z\mid f^n_{k(z)}=0\}$ of such exceptional $z$ is a $\nu_T$--null set.)

By Theorem \ref{thm:HS}, we have $\tau(f^n_k)=\nu_T(\rho([0,\frac{k+1}{2^n}])\backslash \rho([0,\frac{k}{2^n}])).$ Hence,
using~\eqref{eq:logEDn}, we get
\begin{equation}\label{tttkkk}
\log\Delta(|{\rm Exp}_{\mathcal{D}_n}(T)-\lambda|^2+\varepsilon)=\int_{\{|z|\leq\|T\|_{\infty}\}}x_n(z)d\nu_T(z).
\end{equation}
Moreover, we clearly have
$$\log(\varepsilon)\leq x_n(z)\leq \log(\varepsilon+(|\lambda|+\|T\|_{\infty})^2) $$
for every $|z|\le\|T\|_\infty$
and, therefore,
$$\|x_n\|_{\infty}\leq\max\{|\log(\varepsilon)|,|\log(\varepsilon+(|\lambda|+\|T\|_{\infty})^2)|\}.$$

Given $z\in\rho([0,\frac{k+1}{2^n}])\backslash \rho([0,\frac{k}{2^n}])$ with $\tau(f^n_k)\ne0$,
by Theorem \ref{thm:HS} \eqref{hsc},\eqref{hsd} combined with Brown's version of Lidskii theorem \cite{Brown}, we have
$$\frac{\tau(f^n_kTf^n_k)}{\tau(f^n_k)}\in\overline{\rm conv}(\rho([0,\frac{k+1}{2^n}])\backslash \rho([0,\frac{k}{2^n}]))\subset\overline{\rm conv}(\rho((\frac{k}{2^n},\frac{k+1}{2^n}])).$$
Thus,
\begin{equation}\label{ccckkk}
|z-\frac{\tau(f^n_kTf^n_k)}{\tau(f^n_k)}|\leq{\rm diam}(\rho((\frac{k}{2^n},\frac{k+1}{2^n}])))\leq\omega(2^{-n}).
\end{equation}
Combining \eqref{ccckkk} and \eqref{xn def}, we infer that $x_n(z)$ converges to $\log(|z-\lambda|^2+\varepsilon)$
as $n\to\infty$ on a set of full $\nu_T$ measure.
The Dominated Convergence Principle now yields that the right-hand-side of~\eqref{tttkkk} tends to the right-hand-side
of~\eqref{eq:det convergence} as $n\to\infty$.
\end{proof}

Note that one could not first prove Lemma \ref{det convergence} in the case $\lambda=0$ and then infer its assertion in full generality by applying this result to the operator $T-\lambda.$ The reason is that algebra $\mathcal{D}$ for the operator $T$ differs from that fo r the operator $T-\lambda.$

The following proposition is central to this section. It proves Theorem \ref{decomposition} \eqref{decb}.

\begin{prop}\label{br equ} Let $T\in\mathcal{M}$, $\mathcal{D}$ and $\mathcal{D}_n$ be as in Construction \ref{peano}. The Brown measure of the normal operator ${\rm Exp}_{\mathcal{D}}(T)$ equals that of $T.$
\end{prop}
\begin{proof} Fix $\varepsilon>0.$ By Lemma \ref{lem:Enormconv}, we have that ${\rm Exp}_{\mathcal{D}_n}(T)$ converges in norm to ${\rm Exp}_{\mathcal{D}}(T)$ as $n\to\infty.$ Since the Fuglede-Kadison determinant is continuous with respect to the uniform norm topology on the set of invertible elements (see~\cite{FK}), for every $\lambda\in\mathbb{C}$ we have
$$\log\Delta(|{\rm Exp}_{\mathcal{D}_n}(T)-\lambda|^2+\varepsilon)\to\log(\Delta(|{\rm Exp}_{\mathcal{D}}(T)-\lambda|^2+\varepsilon)).$$
On the other hand, it follows from Lemma \ref{det convergence} that
$$\log\Delta(|{\rm Exp}_{\mathcal{D}_n}(T)-\lambda|^2+\varepsilon)\to\int_{\mathbb{C}}\log(|z-\lambda|^2+\varepsilon)d\nu_T(z)$$
for every $\lambda\in\mathbb{C}.$ Hence,
$$\log(\Delta(|{\rm Exp}_{\mathcal{D}}(T)-\lambda|^2+\varepsilon))=\int_{\mathbb{C}}\log(|z-\lambda|^2+\varepsilon)d\nu_T(z)$$
for every $\lambda\in\mathbb{C}.$ Letting $\varepsilon\to0,$ we infer from the Monotone Convergence Principle that
$$\int_{\mathbb{C}}\log(|z-\lambda|)d\nu_{{\rm Exp}_{\mathcal{D}}(T)}(z)=\int_{\mathbb{C}}\log(|z-\lambda|)d\nu_T(z)$$
for every $\lambda\in\mathbb{C}.$ The assertion follows by taking Laplacians of both sides.
\end{proof}

It is tempting to try to infer Theorem \ref{decomposition} from Proposition \ref{br equ} by applying its assertion to the operator $T-{\rm Exp}_{\mathcal{D}}(T).$ This is impossible because the algebra $\mathcal{D}$ for the operator $T$ differs from that for the operator $T-{\rm Exp}_{\mathcal{D}}(T).$

\section{Decomposition}

The following submajorization result is related to the Weyl lemma stating that $\lambda(T)\prec\prec_{\log}\mu(T)$ for every compact operator $T\in B(H)$ (see, e.g., Theorem II.3.1 of \cite{GK1}).

We continue to assume that $\mathcal{M}$ is a diffuse, finite von Neumann algebra and that $\tau$ is a normal, faithful, tracial state
on $\mathcal{M}$.

\begin{lem}\label{lem:anytrace} Let $T\in\mathcal{M}$ and let $p\in\mathcal{M}$ be a projection such that $Tp=pTp.$ Then
$$Tp+(1-p)T\prec\prec_{\log}T.$$
\end{lem}
\begin{proof}
Let $S=Tp+(1-p)T.$ Writing elements of $\mathcal{M}$ as matrices with respect to projections $p$ and $(1-p),$ we have
$$T=\left(\begin{matrix}A&B\\0&C\end{matrix}\right),\qquad S=\left(\begin{matrix}A&0\\0&C\end{matrix}\right),$$
with $A=Tp$ and $C=(1-p)T.$ By Lemma~\ref{boolean lemma}, there exist increasing nets $p_s\le p,$ $0\leq s\leq\tau(p),$ and $q_s\le 1-p,$ $0\leq s\leq\tau(1-p),$ of projections in $\mathcal{M}$ such that $\tau(p_s)=s$, $\tau(q_s)=s$ and such that
$$|C|=\int_0^{\tau({\rm supp}(C))}\mu(s,C)dq_s,\quad |A^*|=\int_0^{\tau({\rm supp}(A))}\mu(s,A)dp_s.$$

Choosing appropriate spectral projections of $A$ and $C$, we see that
for every $t>0,$ there exist $t_1,t_2>0$ such that $t_1+t_2=t$ and such that
\begin{equation}\label{main inequality}
\int_0^t\log(\mu(s,S))ds=\int_0^{t_1}\log(\mu(s,A))ds+\int_0^{t_2}\log(\mu(s,C))ds.
\end{equation}
Set
$$A_0=\int_0^{t_1}\mu(s,A)dp_s,\quad C_0=\int_0^{t_2}\mu(s,C)dq_s$$
and 
$$T_0=\begin{pmatrix}
A_0&B\\
0&C_0
\end{pmatrix},\quad S_0=\begin{pmatrix}
A_0&0\\
0&C_0
\end{pmatrix}.$$
Note that
$$\int_0^t\log(\mu(s,S))ds=\int_0^t\log(\mu(s,S_0))ds.$$

We now claim that $\mu(T_0)\leq\mu(T).$ Indeed,
$$\begin{pmatrix}
A&B\\
0&C
\end{pmatrix}^*\begin{pmatrix}
A&B\\
0&C
\end{pmatrix}=\begin{pmatrix}
A^*A&A^*B\\
B^*A&B^*B+C^*C
\end{pmatrix}\geq\begin{pmatrix}
A&B\\
0&C_0
\end{pmatrix}^*\begin{pmatrix}
A&B\\
0&C_0
\end{pmatrix}$$
and
$$\begin{pmatrix}
A&B\\
0&C_0
\end{pmatrix}\begin{pmatrix}
A&B\\
0&C_0
\end{pmatrix}^*=\begin{pmatrix}
AA^*+BB^*&BC_0\\
C_0B&C_0^2
\end{pmatrix}\geq\begin{pmatrix}
A_0&B\\
0&C_0
\end{pmatrix}\begin{pmatrix}
A_0&B\\
0&C_0
\end{pmatrix}^*.$$
Therefore,
$$\mu(T)=\mu(\begin{pmatrix}
A&B\\
0&C
\end{pmatrix})\geq \mu(\begin{pmatrix}
A&B\\
0&C_0
\end{pmatrix})\geq\mu(\begin{pmatrix}
A_0&B\\
0&C_0
\end{pmatrix})=\mu(T_0).$$

Let now $r=p_{\tau({\rm supp}(C_0))}+q_{\tau({\rm supp}(A_0))}.$ Using~\eqref{det}, we get
$$\int_0^t\log(\mu(s,S))ds=\int_0^t\log(\mu(s,S_0))ds=\log(\Delta_{r\mathcal{M}r}(S_0))$$
and, since $\mu(rT_0r)\le\mu(T_0)\le\mu(T)$, we get
$$\log(\Delta_{r\mathcal{M}r}(rT_0r))\stackrel{\eqref{det}}{=}\int_0^t\log(\mu(s,rT_0r))ds\leq \int_0^t\log(\mu(s,T))ds.$$
It follows now from Theorem~\ref{thm:det nu matrix} that $\Delta_{r\mathcal{M}r}(S_0)=\Delta_{r\mathcal{M}r}(rT_0r)$ and, therefore,
$$\int_0^t\log(\mu(s,S))ds\leq\int_0^t\log(\mu(s,T))ds.$$
\end{proof}

\begin{cor}\label{standard} Let $T\in\mathcal{M}$ and let $p\in\mathcal{M}$ be a projection such that $Tp=pTp.$ Then
$$\Delta(1+\big|Tp+(1-p)T)\big|^2)\le\Delta(1+|T|^2).$$
\end{cor}
\begin{proof} Set $y=\mu(Tp+(1-p)T)$ and $x=\mu(T).$ We may without loss of generality assume $y$ is not identically $0.$ By Lemma \ref{lem:anytrace}, we have $y\prec\prec_{\log}x.$ Set $\alpha$ to be the infimum of the set $y^{-1}(\{0\}).$ Then we must have $\inf x^{-1}(\{0\})\ge\alpha.$ Let $\varepsilon_n=\min\{y(\alpha-\frac1n),x(\alpha-\frac1n)\}$ for all integers $n$ so large that $\frac1n<\alpha.$ Then the functions
$$y_n=\log(\varepsilon_n^{-1}y)\chi_{(0,\alpha-1/n)},\quad x_n=\log(\varepsilon_n^{-1}x)\chi_{(0,\alpha-1/n)},$$
when nonzero, take only values $\ge0.$ Clearly, $y_n\prec\prec x_n.$ Since the function $\Phi_n:z\to\log(1+\varepsilon_n^2e^{2z})$ is convex on $[0,\infty),$ it follows from Theorem \ref{hlp theorem} that
\begin{multline*}
\int_0^{\alpha-1/n}\log(1+y^2(s))ds=\int_0^1\Phi_n(y_n(s))ds-(1-\alpha+\frac1n)\log(1+\varepsilon_n^2)\\
\leq\int_0^{\alpha-1/n}\Phi_n(x_n(s))ds-(1-\alpha+\frac1n)\log(1+\varepsilon_n^2)=\int_0^{\alpha-1/n}\log(1+x^2(s))ds.
\end{multline*}
Letting $n\to\infty,$ we obtain
\begin{multline*}
\int_0^1\log(1+y^2(s))ds=\int_0^{\alpha}\log(1+y^2(s))ds \\
\leq\int_0^{\alpha}\log(1+x^2(s))ds\leq\int_0^1\log(1+x^2(s))ds.
\end{multline*}
Now \eqref{det} finishes the proof.
\end{proof}

The following lemma is easy and we omit the proof.

\begin{lem}\label{monotone triv} If a scalar sequence $\{a_{n,m}\}_{n,m\geq1}$ is decreasing in both arguments, then
$$\lim_{n\to\infty}\lim_{m\to\infty}a_{n,m}=\lim_{m\to\infty}\lim_{n\to\infty}a_{n,m}.$$
\end{lem}

The following lemmas make up the heart of the proof of Theorem \ref{decomposition}.
In the next two lemmas, $\mathcal{D}_n':=\mathcal{M}\cap(\mathcal{D}_n)'$
and $\mathcal{D}':=\mathcal{M}\cap\mathcal{D}'$ mean the relative commutants of the respective algebras in $\mathcal{M}$.

\begin{lem}\label{vspom lemma} Let $p_t\in\mathcal{M},$ $t\in[0,1],$ be an increasing net of projections with $p_0=0$ and $p_1=1$.
Let $m$ and $n$ be positive integers
and let $\mathcal{D}_n$ be the von Neumann subalgebra generated by $p_{k/2^n},$ $0\leq k\leq 2^n.$ If $T\in\mathcal{M}$ and $Tp_t=p_tTp_t$ for all $t\in[0,1]$, then
$$\Delta(|{\rm Exp}_{\mathcal{D}_n'}(T)|^2+\frac1m)\geq \Delta(|{\rm Exp}_{\mathcal{D}_{n+1}'}(T)|^2+\frac1m).$$
\end{lem}
\begin{proof} We let $f_n^k=p_{(k+1)/2^n}-p_{k/2^n},$ $0\leq k<2^n$ and similarly $f_{n+1}^k=p_{(k+1)/2^{n+1}}-p_{k/2^{n+1}},$ $0\leq k<2^{n+1}$ .
For an arbitrary $X\in\mathcal{M}$, we have
\[
{\rm Exp}_{\mathcal{D}_n'}(X)=\sum_{k=0}^{2^n-1}f_n^kXf_n^k,\qquad
{\rm Exp}_{\mathcal{D}_{n+1}'}(X)=\sum_{k=0}^{2^{n+1}-1}f_{n+1}^kXf_{n+1}^k.
\]
Note that $T$ is upper--triangular with respect to the list of projections $(f_n^k)_{0\le k<2^n}$ and $(f_{n+1}^k)_{0\le k<2^{n+1}}$
and, in particular,
$f_{n+1}^{2k}$ is $f_n^kTf_n^k$-invariant and we may write
$$f_n^kTf_n^k=\begin{pmatrix}
f_{n+1}^{2k+1}Tf_{n+1}^{2k+1}&f_{n+1}^{2k}Tf_{n+1}^{2k+1}\\
0&f_{n+1}^{2k}Tf_{n+1}^{2k}
\end{pmatrix}.$$
It follows from Corollary \ref{standard} and Theorem~\ref{thm:det nu matrix} that
\begin{multline*}
\Delta_{f_{n+1}^{2k+1}\mathcal{M}f_{n+1}^{2k+1}}(|f_{n+1}^{2k+1}Tf_{n+1}^{2k+1}|^2+\frac1m)^{\tau(f_{n+1}^{2k+1})}\Delta_{f_{n+1}^{2k}\mathcal{M}f_{n+1}^{2k}}(|f_{n+1}^{2k}Tf_{n+1}^{2k}|^2+\frac1m)^{\tau(f_{n+1}^{2k})} \\
\leq\Delta_{f_n^k\mathcal{M}f_n^k}(|f_n^kTf_n^k|^2+\frac1m)^{\tau(f^k_n)}.
\end{multline*}
It follows now from Theorem~\ref{thm:det nu matrix} that
\begin{multline*}
\Delta(|{\rm Exp}_{\mathcal{D}_n'}(T)|^2+\frac1m)=\prod_{k=0}^{2^n-1}\Delta_{f_n^k\mathcal{M}f_n^k}(|f_n^kTf_n^k|^2+\frac1m)^{\tau(f^k_n)} \\
\geq\prod_{k=0}^{2^{n+1}-1}\Delta_{f_{n+1}^k\mathcal{M}f_{n+1}^k}(|f_{n+1}^kTf_{n+1}^k|^2+\frac1m)^{\tau(f^k_{n+1})}=\Delta(|{\rm Exp}_{\mathcal{D}_{n+1}'}(T)|^2+\frac1m)
\end{multline*}
\end{proof}

The next lemma shows that the Fuglede-Kadison determinant of the operator $T$ coincides with that of its expectation onto the commutant of a nest of invariant projections.

\begin{lem}\label{first exp lemma} Let $p_t\in\mathcal{M},$ $t\in[0,1],$ be an increasing net of projections,
continuous from the right, with $p_0=0$ and $p_1=1$ and let $\mathcal{D}$ be the von Neumann subalgebra generated by $\{p_t\mid t\in[0,1]\}.$ If $T\in\mathcal{M}$ and $Tp_t=p_tTp_t$ for all $t\in[0,1],$ then
\begin{equation}\label{eq:DeltaD'T}
\Delta(T)=\Delta({\rm Exp}_{\mathcal{D}'}(T)).
\end{equation}
\end{lem}
\begin{proof} Let $\mathcal{D}_n$ be the algebra generated by the projections $p_{k/2^n},$ $0\leq k\leq 2^n$.
Using the continuity from the right of $p_t$,
we get $\mathcal{D}'=\bigcap_{n\geq1}\mathcal{D}_n'$ and, by Lemma \ref{exp limit lemma}, we have
$${\rm Exp}_{\mathcal{D}'}(T)=\lim_{n\to\infty}{\rm Exp}_{\mathcal{D}_n'}(T)$$
in the strong operator topology.

By Lemma \ref{vspom lemma}, the quantity $\Delta(|{\rm Exp}_{\mathcal{D}_n'}(T)|^2+\frac1m)$ is decreasing in $n$ and it is, trivially, also decreasing in $m.$ It follows from Lemma \ref{monotone triv} that
\begin{equation}\label{first exp}
\lim_{n\to\infty}\lim_{m\to\infty}\Delta(|{\rm Exp}_{\mathcal{D}_n'}(T)|^2+\frac1m)=\lim_{m\to\infty}\lim_{n\to\infty}\Delta(|{\rm Exp}_{\mathcal{D}_n'}(T)|^2+\frac1m).
\end{equation}

Note that, by $T$-invariance of the projections $p_{k/2^n}$, using them to write $T$ as a block matrix of operators,
yields an upper triangular matrix, and ${\rm Exp}_{\mathcal{D}_n'}(T)$ is obtained by setting the non-diagonal blocks to zero.
To compute the left hand side of \eqref{first exp}, 
Theorem \ref{thm:det nu matrix} yields $\Delta({\rm Exp}_{\mathcal{D}_n'}(T))=\Delta(T)$ and, thereby, we have
$$\lim_{m\to\infty}\Delta(|{\rm Exp}_{\mathcal{D}_n'}(T)|^2+\frac1m)=2\Delta({\rm Exp}_{\mathcal{D}_n'}(T))=2\Delta(T).$$
Letting $n\to\infty,$ we infer that the left hand side of \eqref{first exp} is $2\Delta(T)$.
To compute the right hand side of \eqref{first exp}, by Lemma 2.25 of \cite{HS1}, we have
$$\lim_{n\to\infty}\Delta(|{\rm Exp}_{\mathcal{D}_n'}(T)|^2+\frac1m)=\Delta(|{\rm Exp}_{\mathcal{D}'}(T)|^2+\frac1m).$$
Letting $m\to\infty,$ we infer that the right hand side of \eqref{first exp} is $2\Delta({\rm Exp}_{\mathcal{D}'}(T)).$
Thus, we have~\eqref{eq:DeltaD'T}.
\end{proof}

\begin{lem}\label{brown small bound} Let $T\in\mathcal{M}$ and let $\rho$, $\mathcal{D}$ and $\mathcal{D}_n$ be as in Construction \ref{peano}.
If $\omega$ is the modulus of continuity of $\rho,$ then the Brown measure of $T-{\rm Exp}_{\mathcal{D}_n}(T)$ is supported in the ball of radius $\omega(2^{-n})$ centered
at the origin.
\end{lem}
\begin{proof} The operator $T-{\rm Exp}_{\mathcal{D}_n}(T)$ when written as a matrix with respect to the nonzero projections from the list $(f^n_k)_{k=0}^{2^n-1}$ is upper triangular by Theorem \ref{thm:HS} \eqref{hsb} and has for diagonal entries
\begin{equation}\label{eq:fkdiagelt}
f^n_kTf^n_k-\frac{\tau(f^n_kTf^n_k)}{\tau(f^n_k)}f^n_k.
\end{equation}
Repeating the argument in Lemma \ref{lem:Enormconv} (see \eqref{trace partitioning}), we obtain, when $f^n_k\ne0$, that the Brown measure of $f^n_kTf^n_k$ in $f^n_k\mathcal{M}f^n_k$ is supported in
$\rho((\frac{k}{2^n},\frac{k+1}{2^n}])$ and $\tau(f^n_kTf^n_k)/\tau(f^n_k)$ lies in the convex hull of this set. Hence, the Brown measure of the difference \eqref{eq:fkdiagelt} is supported in the ball centered at the origin of radius $\operatorname{diam}(\rho((\frac{k}{2^n},\frac {k+1}{2^n}]))$, which is no greater than $\omega(2^{-n}).$ Now applying Theorem \ref{thm:det nu matrix} $n$ times, we get that the Brown measure of $T-{\rm Exp}_{\mathcal{D}_n}(T)$ lies in
the ball of radius $\omega(2^{-n})$ centered at the origin.
\end{proof}

The following lemma gives the decomposition of Theorem \ref{decomposition} in a special case.

\begin{lem}\label{comm quasinilpotent}
Let $T\in\mathcal{M}$ and let $\mathcal{D}$ be as in Construction \ref{peano}.
If $T\in\mathcal{D}',$ then $T-{\rm Exp}_{\mathcal{D}}(T)$ is s.o.t.-quasinilpotent.
\end{lem}
\begin{proof} We assume without loss of generality $\|T\|_{\infty}\leq 1/2$.
For every $n\geq0,$ let $\mathcal{D}_n$ be the subalgebra of $\mathcal{D}$ generated by $q_{k/2^n}$,
$0\leq k\leq 2^n$.
Fix $n\in\mathbb{N}$ and a unit vector $\eta\in H.$ By assumption, $T$ commutes with ${\rm Exp}_{\mathcal{D}}(T).$ Since $T-{\rm Exp}_{\mathcal{D}_n}(T)$ and ${\rm Exp}_{\mathcal{D}}(T)-{\rm Exp}_{\mathcal{D}_n}(T)$ commute, we have
$$(T-{\rm Exp}_{\mathcal{D}}(T))^{2m}=\sum_{k=0}^{2m}(-1)^k\binom{2m}{k}({\rm Exp}_{\mathcal{D}}(T)-{\rm Exp}_{\mathcal{D}_n}(T))^{2m-k}(T-{\rm Exp}_{\mathcal{D}_n}(T))^k.$$
Since $\|T\|_{\infty}\leq 1/2,$ it follows that both ${\rm Exp}_{\mathcal{D}}(T)-{\rm Exp}_{\mathcal{D}_n}(T)$ and $T-{\rm Exp}_{\mathcal{D}_n}(T)$ are contractions.

For $k\leq m,$ we have
$$\|({\rm Exp}_{\mathcal{D}}(T)-{\rm Exp}_{\mathcal{D}_n}(T))^{2m-k}(T-{\rm Exp}_{\mathcal{D}_n}(T))^k\eta\|\leq\|{\rm Exp}_{\mathcal{D}}(T)-{\rm Exp}_{\mathcal{D}_n}(T)\|_{\infty}^m.$$
For $k>m,$ we have
$$\|({\rm Exp}_{\mathcal{D}}(T)-{\rm Exp}_{\mathcal{D}_n}(T))^{2m-k}(T-{\rm Exp}_{\mathcal{D}_n}(T))^k\eta\|\leq\|(T-{\rm Exp}_{\mathcal{D}_n}(T))^m\eta\|.$$
Hence, by Lemma \ref{lem:Enormconv}, we have
\begin{equation}\label{eq:TEpow}
\|(T-{\rm Exp}_{\mathcal{D}}(T))^{2m}\eta\|\leq 2^{2m}\max\{\omega(2^{-n})^m,\|(T-{\rm Exp}_{\mathcal{D}_n}(T))^m\eta\|\}.
\end{equation}

By Lemma \ref{brown small bound}, the Brown measure of $T-{\rm Exp}_{\mathcal{D}_n}(T)$ is contained in the ball of radius $\omega(2^{-n})$ centered at $0.$ By the Haagerup--Schultz characterization \eqref{eq:H-S H_r}, for every unit vector $\xi\in H,$ there exists a sequence $\xi_m\to\xi$ such that $\|\xi_m\|=1$ and
$$\limsup_{m\to\infty}\|(T-{\rm Exp}_{\mathcal{D}_n}(T))^m\xi_m\|^{1/m}\leq\omega(2^{-n}).$$
Hence, there exists $M$ (depending on $n$), such that
$$\|(T-{\rm Exp}_{\mathcal{D}_n}(T))^m\xi_m\|\leq (2\omega(2^{-n}))^m,\quad m>M.$$
Substituting $\xi_m$ for $\eta$ in \eqref{eq:TEpow}, we obtain
$$\|(T-{\rm Exp}_{\mathcal{D}}(T))^{2m}\xi_m\|^{1/m}\leq 8\omega(2^{-n}),\quad m>M.$$
Since $\xi$ was any unit vector, it follows from the characterization \eqref{eq:H-S H_r} of the Haagerup--Schultz subspace that the Brown measure of $(T-{\rm Exp}_{\mathcal{D}}(T))^2$ is supported in the ball centered at $0$ of radius $8\omega(2^{-n}).$ Letting $n\to\infty,$ we obtain that the Brown measure of $T-{\rm Exp}_{\mathcal{D}}(T)$ is $\delta_0.$
\end{proof}

\begin{proof}[Proof of Theorem \ref{decomposition}] Let $q_t,$ $t\in[0,1],$ and $\mathcal{D}$ be as in Construction \ref{peano}. Applying Lemma \ref{first exp lemma} to the operators $Tq_t-\lambda$ and $(1-q_t)T-\lambda,$ we obtain
\begin{align*}
\Delta(Tq_t-\lambda)&=\Delta({\rm Exp}_{\mathcal{D'}}(Tq_t)-\lambda), \\
\Delta((1-q_t)T-\lambda)&=\Delta({\rm Exp}_{\mathcal{D}'}((1-q_t)T-\lambda)).
\end{align*}
Hence, by \eqref{br def}, the Brown measure of the operator $Tq_t$ (respectively, of $(1-q_t)T$) equals that of ${\rm Exp}_{\mathcal{D}'}(T)q_t$ (respectively, of $(1-q_t){\rm Exp}_{\mathcal{D}'}(T)$). In particular, the projection $q_t$ is the Haagerup--Schultz projection for ${\rm Exp}_{\mathcal{D}'}(T)$ associated to the set $\rho([0,t]).$ By definition of $\mathcal{D}',$ the operator ${\rm Exp}_{\mathcal{D'}}(T)$ commutes with every $q_t.$ Hence, ${\rm Exp}_{\mathcal{D}'}(T)$ satisfies the hypotheses of Lemma \ref{comm quasinilpotent}. Since $\mathcal{D}\subseteq\mathcal{D}'$ we have ${\rm Exp}_{\mathcal{D}}({\rm Exp}_{\mathcal{D}'}(T))={\rm Exp}_{\mathcal{D}}(T)$ and by application of this lemma we infer that the Brown measure of the operator ${\rm Exp}_{\mathcal{D'}}(T)-{\rm Exp}_{\mathcal{D}}(T)$ is $\delta_0.$

Applying Lemma \ref{first exp lemma} to the operator $T-{\rm Exp}_{\mathcal{D}}(T)-\lambda,$ we obtain
$$\Delta(T-{\rm Exp}_{\mathcal{D}}(T)-\lambda)=\Delta({\rm Exp}_{\mathcal{D'}}(T)-{\rm Exp}_{\mathcal{D}}(T)-\lambda).$$
Hence, by \eqref{br def}, the Brown measure of the operator $T-{\rm Exp}_{\mathcal{D}}(T)$ equals that of ${\rm Exp}_{\mathcal{D'}}(T)-{\rm Exp}_{\mathcal{D}}(T)$, i.e., it is $\delta_0$.

Now letting $Q=T-{\rm Exp}_{\mathcal{D}}(T)$ and $N={\rm Exp}_{\mathcal{D}}(T)$, we have $T=N+Q$ with $N$ normal
and $Q$ s.o.t.-quasinilpotent. This proves Theorem \ref{decomposition} \eqref{decc}. Theorem \ref{decomposition} \eqref{decb} is already proved in Proposition \ref{br equ}.
\end{proof}

\section{A Weyl inequality in the II$_1$--setting}

In the following assertion, $\log_+(t):=\max(\log(t),0)$.
The analogue of the following assertion for Hardy-Littlewood submajorization is well-known.

\begin{lem}\label{logmaj another} We have $B\prec\prec_{\log}A$ if and only if for all $t>0$ we have
\begin{equation}\label{eq:logAtBt}
\tau(\log_+(\frac{|B|}{t}))\leq\tau(\log_+(\frac{|A|}{t})).
\end{equation}
\end{lem}
\begin{proof} We prove the if part first. For a given $u\in(0,1),$ set $t=\mu(u,A).$ We have
$$\int_0^u\log(\frac{\mu(s,B)}{t})ds\leq\tau(\log_+(\frac{|B|}{t}))\leq\tau(\log_+(\frac{|A|}{t}))=\int_0^u\log(\frac{\mu(s,A)}{t})ds.$$
Hence, $B\prec\prec_{\log}A.$

We now prove the only if part. If $t>\|B\|_{\infty},$ then \eqref{eq:logAtBt} holds trivially, while if $t\le\|B\|_{\infty}$ then
we have
\begin{multline*}
\tau(\log_+(\frac{|B|}{t}))=\int_0^{d_{|B|}(t)}\log(\frac{\mu(s,B)}{t})ds\leq\int_0^{d_{|B|}(t)}\log(\frac{\mu(s,A)}{t})ds \\
\leq\int_0^{d_{|A|}(t)}\log(\frac{\mu(s,A)}{t})ds=\tau(\log_+(\frac{|A|}{t})).
\end{multline*}
\end{proof}

\begin{lem}\label{stamp} Let $\mathcal{M}$ be a finite von Neumann algebra equipped with a faithful normal tracial state $\tau.$ If $0\leq A,B\in\mathcal{M}$ are such that $B\prec\prec_{\log}A,$ then $B+1\prec\prec_{\log}A+1.$
\end{lem}
\begin{proof}
By enlarging $\mathcal{M}$ if necessary, we may without loss of generality assume it is diffuse.
Referring to Lemma \ref{boolean lemma}, we infer the existence of increasing nets $p_s,q_s$ $0\leq s\leq1,$ of projections in $\mathcal{M}$ such that $\tau(p_s)=s,$ $\tau(q_s)=s$ and such that
$$B=\int_0^1\mu(s,B)dq_s,\quad A=\int_0^1\mu(s,A)dp_s.$$
Fix $t<\tau({\rm supp}(B))$ and define the operators
$$B_t=\int_0^t\mu(s,B)dq_s,\quad C_t=\int_0^t\mu(s,B)dp_s,\quad A_t=\int_0^t\mu(s,A)dp_s.$$
It is clear that $\mu(B_t)=\mu(C_t),$ $C_t\prec\prec_{\log} A_t$ and the operators $A_t,C_t$ are in the von Neumann algebra $p_t\mathcal{M}p_t$.
Note that we have $\tau({\rm supp}(B_t))=\tau({\rm supp}(C_t))\le\tau({\rm supp}(A_t))$.
Since $t<\tau({\rm supp}(B)),$ it follows that $A_t$ and $C_t$ are invertible
and $\log(A_t),\log(C_t)\in p_t\mathcal{M}p_t$.
Let $r=\max\{\|A_t^{-1}\|_{\infty},\|C_t^{-1}\|_{\infty}\}$.
By Lemma~\ref{lem:log invertibles}, from $C_t\prec\prec_{\log}A_t$ we get
$rC_t\prec\prec_{\log}rA_t$ and $\log(r C_t)\prec\prec\log(rA_t)$.
Since the mapping $\Phi:z\to\log(1+r^{-1}e^z)$ is convex, it follows from Theorem \ref{hlp theorem} that
\[
\int_0^t\log(1+\mu(s,B))ds=\tau(\Phi(\log(rC_t)))\leq\tau(\Phi(\log(rA_t)))=\int_0^t\log(1+\mu(s,A))ds.
\]
Since $t<\tau({\rm supp}(B))$ is arbitrary, the assertion follows.
\end{proof}

Now we prove Theorem~\ref{thm:WeylII1} and some more.
\begin{thm} Let $\mathcal{M}$ be a
diffuse, finite von Neumann algebra and let $\tau$ be a normal, faithful, tracial state on $\mathcal{M}$.
Let $T\in\mathcal{M}$.
If $N\in\mathcal{M}$ is the normal operator constructed by means of Theorem \ref{decomposition}, then
\begin{enumerate}[{\rm (a)}]
\item\label{kika} $N\prec\prec_{\log}T.$
\item\label{kikb} for every increasing real valued function $\Phi$ on $[0,\infty)$ so that $\Phi\circ\exp$ is convex, we have
\begin{equation}\label{eq:Weyl}
\int\Phi(|z|)\,d\mu_T(z)=\tau(\Phi(|N|))\leq\tau(\Phi(|T|)).
\end{equation}
\end{enumerate}
\end{thm}
\begin{proof} Since the Brown measure of $N$ equals the distribution of $N$ and equals the Brown measure of $T$, we have
$$\tau(\log_+(\frac{|N|}{t}))=\int_{|z|\geq t}\log(\frac{|z|}{t})d\nu_N(z)=\int_{|z|\geq t}\log(\frac{|z|}{t})d\nu_T(z).$$
It follows from Lemma 2.20 in \cite{HS1} that
$$\int_{|z|\geq t}\log(\frac{|z|}{t})d\nu_T(z)\leq\tau(\log_+(\frac{|T|}{t})).$$
Therefore,
$$\tau(\log_+(\frac{|N|}{t}))\leq\tau(\log_+(\frac{|T|}{t}))$$
Assertion \eqref{kika} follows now from Lemma \ref{logmaj another}.

The equality in~\eqref{eq:Weyl} holds since $\mu_T=\mu_N$ and $N$ is normal.
Assertion \eqref{kika} together with Lemma \ref{stamp} and homogeneity of logarithmic submajorization shows that
$$n|N|+1\prec\prec_{\log}n|T|+1,\quad n\geq1.$$
Now Lemma~\ref{lem:log invertibles} yields $\log(n|N|+1)\prec\prec\log(n|T|+1)$.
Since $\Phi\circ\exp$ is convex on $\mathbb{R}$, so is the function $\Phi(\frac1n\exp(\cdot))$ and
from Theorem \ref{hlp theorem} we get
$$\tau(\Phi(|N|+\frac1n)))\leq\tau(\Phi(|T|+\frac1n))),\quad n\geq 1.$$
We now infer from the Monotone Convergence Principle that
$$\tau(\Phi(|N|))=\lim_{n\to\infty}\tau(\Phi(|N|+\frac1n))\leq\lim_{n\to\infty}\tau(\Phi(|T|+\frac1n))=\tau(\Phi(|T|)).$$
This proves \eqref{kikb}.
\end{proof}

\end{document}